\documentclass[a4paper,12pt]{amsart}
\usepackage{fullpage}
\usepackage[utf8]{inputenc}
\usepackage{array}
\usepackage{cite}
\usepackage{amssymb,amsthm,amsmath,mathrsfs,amsfonts}
\usepackage{fancyhdr,color}
\usepackage{amsmath,amssymb,amsthm,graphicx,enumerate, fancyhdr, microtype}
\usepackage{amssymb, amsmath,amsthm,xcolor,bm,hyperref,color,tikz,euscript}
\usepackage{xcolor,bm,hyperref,color,tikz,euscript,sansmath,ytableau, mathtools}
\usetikzlibrary{3d,calc}
\usetikzlibrary{positioning,matrix,arrows,decorations.pathmorphing}
\usepackage{amssymb,amsthm,amsmath,mathrsfs}
\usepackage{fancyhdr,color}
\usepackage{graphicx,enumerate, fancyhdr, microtype}
\usepackage{xcolor,bm,color,tikz,euscript,sansmath,ytableau}
\usepackage{tikz-cd}
\usepackage{subfig}

\usepackage[all]{xy}

\usepackage{hyperref}
\usepackage{url, hypcap}
\definecolor{darkblue}{RGB}{0,0,160}
\hypersetup{
colorlinks,%
citecolor=darkblue,%
filecolor=black,%
linkcolor=darkblue,%
urlcolor=darkblue
}

\usepackage[colorinlistoftodos]{todonotes}
\usepackage{cleveref}
\crefname{thm}{theorem}{theorems}
\Crefname{thm}{Theorem}{Theorems}

\tikzstyle{invisivertex} = [black, shape=rectangle, minimum size=0pt, inner sep=3pt]
\tikzstyle{moebius} = [draw, violet!80!white, shape=circle, minimum size=0pt, inner sep=2pt]

\newcommand{\A}{\mathcal{A}}

\newcommand{\LL}{\mathcal{L}}

\newcommand{\VG}{\operatorname{V\ \hspace*{-6pt} G}}
\newcommand{\FP}{\operatorname{O\ \hspace*{-5pt} R}}

\newcommand{\C}{\mathbb{C}}
\newcommand{\R}{\mathbb{R}}

\newcommand{\RR}{\mathbf{R}}
\newcommand{\Anti}{\mathbf{A}}

\newcommand{\ceil}{\mathrm{ceil}}

\newcommand{\ideal}{{\ideals}}
\newcommand{\ideals}{{\mathcal I}}
\newcommand{\codim}{{\operatorname{codim}}}

\newcommand{\Hilb}{\mathrm{Hilb}}
\newcommand{\Poin}{\mathrm{Poin}}
\newcommand{\whit}[2]{\thinspace c_{#1}(#2)\thinspace} 
\newcommand{\GGG}{{\mathcal G}}
\newcommand{\grr}{{\mathfrak{gr}}}

\newcommand{\Shi}{{\sf Shi}}
\newcommand{\Nar}{{\sf Nar}}

\newcommand{\one}{{e}}
\newcommand{\Inv}{\operatorname{Inv}}

\newcommand{\Dfn}[1]{\emph{\bfseries #1}}
\renewcommand{\th}{${}^{\operatorname{th}}$}

\newcommand{\changes}[1]{#1}
\newcommand{\init}{\mathrm{in}}
\newtheorem{thm}{Theorem}[section]
\newtheorem{clm}[thm]{Claim}

\newtheorem{coro}[thm]{Corollary}
\newtheorem{prop}[thm]{Proposition}
\newtheorem{lem}[thm]{Lemma}

\theoremstyle{definition}

\newtheorem{ex}[thm]{Example}
\newtheorem{rem}[thm]{Remark}
\newtheorem{note}[thm]{Note}

\title{Shi arrangements restricted to Weyl cones}

\author[G.~Dorpalen-Barry]{Galen Dorpalen-Barry}
\address[G.~Dorpalen-Barry]{Department of Mathematics, Texas A\&M University, USA}
\email{dorpalen-barry@tamu.edu}

\author[C.~Stump]{Christian Stump}
\address[C.~Stump]{Fakultät für Mathematik, Ruhr-Universit\"at Bochum, Germany}
\email{christian.stump@rub.de}

\subjclass[2010]{Primary 20F55; Secondary 52C35}

\thanks{
  We would like to thank Frédéric Chapoton and Vic Reiner for useful references, interesting discussions, and feedback on a draft of this manuscript.
  We also thank Raman Sanyal for helpful comments.
  \changes{We are grateful to the two anonymous referees for their detailed feedback; their comments significantly improved the readability of the manuscript.}
  Both authors are supported by the DFG Heisenberg grant STU 563/\{4--6\} ``Noncrossing phenomena in Algebra and Geometry''.
}

\begin{document}

\begin{abstract}
  We consider the restrictions of Shi arrangements to Weyl cones, their relations to antichains in the root poset, and their intersection posets.
  For any Weyl cone, we provide bijections between regions, flats intersecting the cone, and antichains of a naturally-defined subposet of the root poset.
  This gives a refinement of the parking function numbers via the Poincar\'e polynomials of the intersection posets of all Weyl cones.
  Finally, we interpret these Poincar\'e polynomials as the Hilbert series of three isomorphic graded rings.
  One of these rings arises from the Varchenko-Gel'fand ring, another is the coordinate ring of the vertices of the order polytope of a subposet of the root poset, and the third is purely combinatorial.
\end{abstract}

\maketitle


\section{Introduction and main results}

Shi arrangements are well-studied hyperplane arrangements associated to an irreducible finite Weyl group.
These were initially defined by Shi to study \emph{Kazhdan-Lusztig cells} of affine Weyl groups~\cite{shi-sign-types} and since then have appeared in many algebraic, geometric, and combinatorial contexts.
We refer to~\cite[\S 5.1.4]{armstrong} and~\cite{fishel} for their history in the context of finite Weyl groups.
Among many, two classical results state that
\begin{itemize}
  \item the number of regions is the \Dfn{$W$-parking function number} $(h + 1)^{\ell}$ where $h$ is the Coxeter number and $\ell$ is the rank~\cite[Theorem~8.1]{shi-sign-types}, and
  \item the number of dominant regions is the \Dfn{$W$-Catalan number} $\prod_{i=1}^\ell (d_i + h)/{d_i}$, where $d_1,\dots,d_\ell$ are the invariant degrees~\cite[Theorem 1.1]{yoshinaga}.
\end{itemize}
In this paper, we study interpretations and refinements of these two (and of related) formulas in terms of the intersection poset of the Shi arrangement.
In the final section, we construct three rings and show that they are isomorphic for Weyl cones of Shi arrangements: the Varchenko--Gel'fand ring of a cone (in the sense of \cite{dorpalenbarry}), the coordinate ring of the vertices of the order polytope of the root poset, and the order ring of a poset (in the sense of \cite{chapoton}). 

\subsection{Background}

We start by briefly fixing the necessary notation.
We refer to~\cite{humphreys} for further background on root systems and to~\cite{fishel} for a more detailed introduction to Shi arrangements.
In \Cref{sec:longexample}, we illustrate the main results with a detailed example in type~$B_2$.

\medskip

Let $\Delta \subseteq \Phi^+ \subseteq \Phi \subset V$ be an irreducible crystallographic root system with a given choice of \Dfn{simple roots}~$\Delta$, \Dfn{positive roots}~$\Phi^+$ and \Dfn{negative roots} $\Phi^- = -\Phi^+$ inside a Euclidean space~$V$ of dimension~$\ell$ with inner product $\langle\cdot,\cdot\rangle$.
The corresponding Weyl group $W = W(\Phi)$ is the subgroup of the orthogonal group~$O(V)$ generated by the reflections along the hyperplanes orthogonal to the roots in~$\Phi$.
We assume the action of~$W$ on~$V$ to be irreducible, meaning that~$W$ does not fix any non-zero proper subspace of~$V$.
Then the \Dfn{Shi arrangement} is
\[
  \Shi(\Phi^+) = \Shi_0(\Phi^+) \cup \Shi_1(\Phi^+)\,,
\]
where
\[
  \Shi_i(\Phi^+) = \big\{ H_{\beta,i} = \{ v \in V \mid \langle v, \beta\rangle = i\} \mid \beta \in \Phi^+ \big\} \,.
\]
We observe that $\Shi_0(\Phi^+) = \A(W)$ is the \Dfn{reflection arrangement} of the Weyl group~$W$.
In particular, it does not depend on the choice of simple and positive roots.
The \Dfn{dominant cone} in~$\A(W)$ (also known as \emph{fundamental chamber}) is the (open) cone of~$V$ defined by
\[
  C = \{ v \in V \mid \langle v,\beta\rangle > 0 \text{ for all } \beta \in \Phi^+ \}
\]
and for any $w\in W$, the map $w \mapsto wC$ is a bijection between elements in~$W$ and cones of $\A(W)$, that is, the (open) connected components of the complement $V \setminus \bigcup_{H\in \A(W)} H$.
These cones have the explicit description
\[
  wC = \bigcap_{\beta\,\in\, \Inv(w)} H_{\beta,0}^- \cap \bigcap_{\beta\,\in\, \Phi^+\setminus \Inv(w)} H_{\beta,0}^+\,,
\]
where $\Inv(w) = \Phi^+ \cap w\Phi^-$ is the (left-)\Dfn{inversion set} of $w \in W$ and
\[
  H_{\beta,i}^+ = \{ v \in V \mid \langle v, \beta\rangle > i\} \text{ and } H_{\beta,i}^- = \{ v \in V \mid \langle v, \beta\rangle < i\}
\]
are the open halfspaces defined by the hyperplane $H_{\beta,i}$.
Each cone~$wC$ contains (possibly multiple) Shi regions.
This suggests the following refinement
\[
  (h+1)^\ell = \# \RR\big(\Shi(\Phi^+)\big) = \sum_{w \in W} \# \RR_w\,,
\]
where $\RR\big(\Shi(\Phi^+)\big)$ is the set of Shi regions and
\[
  \RR_w = \big\{ R \in \RR(\Shi(\Phi^+)) \mid R \subset wC\big\}
\]
is the set of Shi regions inside the cone~$wC$.
The Shi regions $\RR_\one$ for the identity element $\one \in W$ inside the dominant cone~$C$ are the \Dfn{dominant regions} and are counted by the Catalan numbers.
These are in bijection with certain subsets of positive roots, namely antichains in the \Dfn{root poset} which is the poset on positive roots~$\Phi^+$ given by the cover relations $\beta \prec \gamma$ for $\gamma - \beta \in \Delta$~\cite{athanasiadis-catalan-nums}.

\subsection{Main results}

Our first result generalizes this bijective correspondence to Shi regions inside any cone~$wC$.
The \Dfn{ceiling set} $\ceil(R) \subseteq\Phi^+$ of a Shi region~$R$ is the set of positive roots~$\beta$ such that $H_{\beta,1}$ is a facet-defining hyperplane of~$R$ with $R \subset H_{\beta,1}^-$.
This is,~$\beta$ is a ceiling of~$R$ if $R \subseteq H_{\beta,1}^-$ and $\overline R \cap H_{\beta,1}$ is a face of dimension $\ell-1$ of the closure~$\overline R$ of the region~$R$.
The following gives a bijection between $\RR_w$ and the set of antichains in the subposet $\Phi^+ \setminus \Inv(w^{-1})$ of the root poset,
\[
  \Anti_w = \big\{ A \subset \Inv(w^{-1}) \mid \beta,\gamma \text{ incomparable for any } \beta,\gamma \in A \text{ with } \beta \neq \gamma \big\}\,.
\]
An \Dfn{order ideal} in~$\Phi^+$ is a subset $I \subseteq \Phi^+$ such that $\beta \prec \gamma \in I$ in root poset order implies $\beta \in I$.
To an antichain $A \in \Anti_w$, we associate the corresponding order ideal in $\Phi^+\setminus \Inv(w^{-1})$ generated by~$A$,
\[
  \ideal_w(A) = \big\{ \gamma \in \Phi^+\setminus\Inv(w^{-1}) \mid \gamma \prec \beta \text{ for some } \beta \in A \big\}\,.
\]

\begin{thm}
\label{main-thm:chamber-bijection}
  The map
  \begin{align*}
    \varphi_w : \RR_w &\longrightarrow  \Anti_w \\
    R &\longmapsto w^{-1}\big(\ceil(R)\big)
  \intertext{
  is a bijection with inverse
  }
    \varphi_w^{-1} : \Anti_w &\longrightarrow \RR_w \\
    A &\longmapsto \left\{
      v \in V ~\middle\vert~
      \begin{matrix}
      0 < \langle v, \gamma \rangle < 1 \text{ for } \gamma\in w\big(\ideal_w(A)\big)\\[3pt]
      1 < \langle v, \gamma \rangle\hspace*{22pt} \text{ for } \gamma\notin w\big(\ideal_w(A)\big)
      \end{matrix}
    \right\}\,.
  \end{align*}
\end{thm}

This theorem solves a problem originally proposed by Athanasiadis-Linusson in~\cite[Section 4]{athanasiadis-linusson}, but we will see that it follows directly from previous work of Athanasiadis~\cite{athanasiadis-catalan-nums} and Armstrong-Reiner-Rhoades~\cite{Armstrong-Reiner-Rhoades}.
Indeed, without making it explicit, this correspondence was already implicitly considered in~\cite{Armstrong-Reiner-Rhoades}, see~\Cref{lem:inversions} below.

\medskip

The main result of this paper is the following, which gives a bijection between antichains in $\Anti_w$ (and thus also of the Shi regions in $\RR_w$) and elements in the \Dfn{intersection poset}
\[
  \LL_w = \big\{ X \in \LL\big(\Shi(\Phi^+)\big) \mid X \cap wC \neq \emptyset \big\},
\]
where $\LL\big(\Shi(\Phi^+)\big)$ is the \Dfn{intersection poset} of the Shi arrangement given by the set of (nonempty) intersections of hyperplanes of $\Shi(\Phi^+)$.
Observe that $\LL_w$ is an intersection poset of a cone as considered in~\cite{dorpalenbarry}.

\medskip

While having a similar flavor as \Cref{main-thm:chamber-bijection}, the following result is---to the best of our knowledge---an interpretation of antichains of the root poset that has not appeared in the literature.

\begin{thm}
\label{main-thm:intersection-bijection}
  The map
  \begin{align*}
    \psi_w : \LL_w &\longrightarrow  \Anti_w \\
    X &\longmapsto \big\{w^{-1}(\beta) \mid \beta\in E,\ X \subseteq H_{\beta,1}\big\}
  \intertext{
    is a bijection with inverse
  }
    \psi_w^{-1} : \Anti_w &\longrightarrow \LL_w \\
    A &\longmapsto \bigcap_{\beta\in wA}H_{\beta,1}\,.
  \end{align*}
\end{thm}

\begin{rem}[Connection to nonnesting partitions]
  Antichains in the root poset (that is, elements of $\Anti_e$ for the identity element~$e \in W$) are called \emph{nonnesting partitions}.
  \Cref{main-thm:intersection-bijection} shows that the map $A \mapsto \bigcap_{\beta\in A}H_{\beta,1}$ is injective.
  Athanasiadis and Reiner showed injectivity of the map $A \mapsto \bigcap_{\beta \in A} H_{\beta,0}$, see~\cite[Corollary~6.2]{athanasiadis-reiner}.
  Their result establishes the embedding of the \emph{nonnesting flats} into the intersection poset $\LL\big(\A(W)\big)$.
  Thus the above result may be seen as a \emph{lift} of the nonnesting flats into the affine intersection poset $\LL\big(\Shi_1(\Phi^+)\big)$ whose image are exactly the flats cutting through the dominant cone~$C$ (that is, flats in~$\LL_e$).
  Our proof of injectivity does \emph{not} rely on their argument, nor can we deduce their result from ours without reproducing their argument.
\end{rem}

\begin{rem}[Similarities to a poset of Biane and Josuat-Verg\'es]
In~\cite[Definition~4.1]{biane-josuat-verges}, Biane and Josuat-Verg\'es consider the partial order on \emph{noncrossing partitions}, which depends on a choice of a \emph{Coxeter element}, given by the intersection of the \emph{absolute order} and the \emph{Bruhat order}.
The rank sizes of their posets are also given by the Narayana numbers and all lower intervals are also Boolean lattices, compare \Cref{main-thm:set-cardinality-equality}.
In type~$A_4$, there is no choice of Coxeter element for which that poset is isomorphic to the interesection poset~$\LL_e$.
\end{rem}

Our third result is proved as part of the proof of \Cref{main-thm:intersection-bijection} and concerns the structure of $\LL_w$ as a poset ordered by reverse inclusion.
This poset is an order ideal of the intersection poset $\LL\big(\Shi(\Phi^+)\big)$ of the full arrangement and thus a meet-semilattice ranked by codimension, for which every lower interval $[V,X]\subseteq \LL_w$ is a geometric lattice.
The following theorem strengthens this property and shows that the M\"{o}bius function values of lower intervals are all equal to~$\pm 1$, implying that $\LL_w$ is Eulerian.

\begin{thm}
\label{main-thm:set-cardinality-equality}
  All lower intervals of the poset~$\LL_w$ are Boolean.
  Thus, $\mu(V,X) = (-1)^{\codim(X)}$ for all $X\in \LL_w$ and in particular
  \[
  \#\LL_w = \# \RR_w.
  \]
\end{thm}

The next results follow immediately from the above.
We state them here, as they interpret and generalize previously-studied combinatorial properties of Shi arrangements.
The \Dfn{Poincar\'{e} polynomial} of the cone\footnote{Although we give a combinatorial description here, this really is the Poincaré polynomial of a topological space.
The space is described in \cite[Theorem 1.1]{dbpw} and the connection to posets is described in \cite[Main Theorem]{dorpalenbarry}.} $wC$ is
\[
  \Poin(wC,t) =  \sum_{X\in \LL_w}|\mu(V,X)| ~ t^{\codim{X}} = \sum_{k \geq 0} \whit{k}{wC}\cdot t^k,
\]
and the coefficients $\whit{k}{wC}$ are its \Dfn{Whitney numbers}%
.
Thus \Cref{main-thm:set-cardinality-equality} and \Cref{main-thm:chamber-bijection} give the following corollary.

\begin{coro}
\label{main-thm:whitney-numbers}
  The $k$\th Whitney number of~$wC$ is
  \begin{align*}
    \whit{k}{wC}
    &= \# \big\{ A\in\Anti_w \mid \# A = k \big\} \\
    &= \# \big\{ R \in \RR_w \mid \# \ceil(R) = k \big\}.
  \intertext{
    In particular,
  }
    \Poin(wC,1) &= \#\Anti_w = \#\RR_w
  \intertext{
    and summing over all $w \in W$ yields
  }
    \sum_{w \in W} \Poin(wC,1) &= (h+1)^\ell\,.
  \end{align*}
\end{coro}

%

The \Dfn{$W$-Narayana numbers} are a refinement of the $W$-Catalan numbers and one way to define these is by counting antichains in the root poset by cardinality, i.e.,
\[
  \Nar^k(\Phi^+) = \# \big\{ A\in\Anti_\one \mid \# A = k \big\}\,.
\]
\Cref{main-thm:whitney-numbers} thus gives another interpretation of the Narayana numbers in terms of the Whitney numbers of the dominant cone.

\begin{coro}
  We have
  \(
    \Nar^{k}(\Phi^+) = \whit{k}{C}\,.
  \)
\end{coro}

\begin{rem}[$m$-eralizations]
  The Shi arrangement has a well-known \emph{$m$-eralization} given by
  \[
    \Shi^{(m)}(\Phi^+) = \bigcup_{-m < k \leq m} \Shi_k(\Phi^+)\,.
  \]
  Its region count is the \Dfn{Fuss-parking function number} $(mh+1)^\ell$ and its dominant region count is the \Dfn{Fuss-Catalan number} $\prod_{i=1}^\ell (d_i + mh)/{d_i}$ as given in~\cite[Theorem~1.1]{yoshinaga}.
  There is also an $m$-eralization of antichains by Athanasiadis~\cite{athanasiadis-catalan-nums}.
  The results presented here do not immediately suggest an $m$-eralized bijection, as we can already see in rank $2$.
  In type $A_2$ with $m=2$, the dominant cone~$C$ of $\Shi^{(2)}(\Phi^+)$ is
  
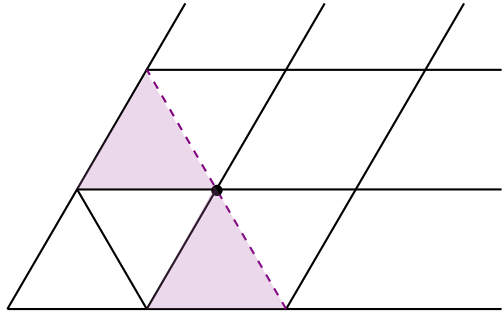
\begin{center}
  \begin{tikzpicture}[scale=1.3]
    \path[-] (0,0) edge [thick] (  0:5cm);
    \path[-] ($(0,0) + (0.7,1.22) $) edge [thick] ($(0  :5cm) + (0,1.22) $);
    \path[-] ($(0,0) + (1.4,2.44) $) edge [thick] ($(0  :5cm) + (0,2.44) $);

    \path[-] (0,0) edge [thick] ( 60:3.6cm);
    \path[-] ($(0,0) + (1.41,0) $) edge [thick] ($(60:3.6cm) + (1.41,0) $);
    \path[-] ($(0,0) + (2.82,0) $) edge [thick] ($(60:3.6cm) + (2.82,0) $);

    \path[-] ($(0,0) + (1.41,0) $) edge [thick] ($(120:1.41cm) + (1.41,0) $);
    \path[-] ($(0,0) + (2.82,0) $) edge [color=violet, thick, dashed] ($(120:2.82cm) + (2.82,0) $);

    \node at (2.12,1.21)[circle,fill,inner sep=3pt]{};

    \fill[violet!60,nearly transparent] (1.41,0) -- (2.82,0) -- (2.1,1.25) -- cycle;
    \fill[violet!60,nearly transparent] (0.7,1.23) -- (2.1,1.23) -- (1.41,2.41) -- cycle;
  \end{tikzpicture}
  \end{center}
  The two shaded regions have the same single-element ceiling set, $H_{\alpha+\beta,2}$ (dashed).
  Moreover, the three hyperplanes
  \[
    H_{\alpha,1} \cap H_{\beta,1} \cap H_{\alpha+\beta,2} = \{\bullet\}
  \]
  meet at a zero-dimensional flat with absolute M\"obius function value~$|\mu(V,\bullet)| = 2$.
  Thus we can have neither a bijection between regions in the dominant cone and certain antichain generalizations defined via ceiling sets (indeed, the bijection in~\cite{athanasiadis-catalan-nums} uses more information than only ceilings), nor can we have a bijection between flats cutting through~$C$ and these dominant regions.
  In particular, the above flat is the unique flat with absolute M\"obius function value not equal to~$1$ and the number of flats inside~$C$ is~$11$ while there are~$12$ dominant regions.

  Although we do not define them here, the Fuss-Catalan numbers are refined by \emph{Fuss-Narayana numbers} and one can check that the Whitney number distribution matches Fuss-Narayana distribution in rank~$2$ (for all~$m$).
  This pattern fails in higher ranks as, for example, the Poincar\'{e} polynomial of the dominant cone in type $A_3$ with $m=2$ is $1 + 12t + 29t^2 + 13t^3$ while the corresponding Fuss-Narayana polynomial is~$1 + 12t + 28t^2 + 14t^3$.
\end{rem}

After a detailed example in Section \ref{sec:longexample}, the remainder of this paper is organized as follows. 
In \Cref{sec:deleted1}, we extend and prove all above main results for the dominant cone~$C$ and \emph{certain} subarrangements of the Shi arrangement, namely those containing the reflection arrangement.
We then deduce the main results from that setup in \Cref{sec:deduction}.
In \Cref{sec:algebra}, we finally interpret the above Poincar\'e polynomials as the Hilbert series of two isomorphic graded rings, one arising from the Varchenko-Gel'fand ring and another, which we call the order ring of a finite poset since it turns out to be naturally associated to the order polytope.

\subsection{The results in action}
\label{sec:longexample}

We close this introduction with a detailed illustration of the main results in the example of type~$B_2$.
The simple and positive roots are
\[
  \Delta = \{\alpha, \beta\} \subseteq \Phi^+ = \{\alpha, \beta, \alpha+\beta, 2\alpha+\beta\}.
\]
We write $s = s_\alpha$ and $t = s_\beta$ to be the reflections orthogonal to~$\alpha$ and to~$\beta$, respectively.
One may realize this root system in the standard basis $\{e_1,e_2\}$ of $\R^2$ by $\alpha = e_2, \beta = e_1 - e_2$.
In this realization, the arrangements $\A(W)$ (left) and $\Shi(\Phi^+)$ (right) are
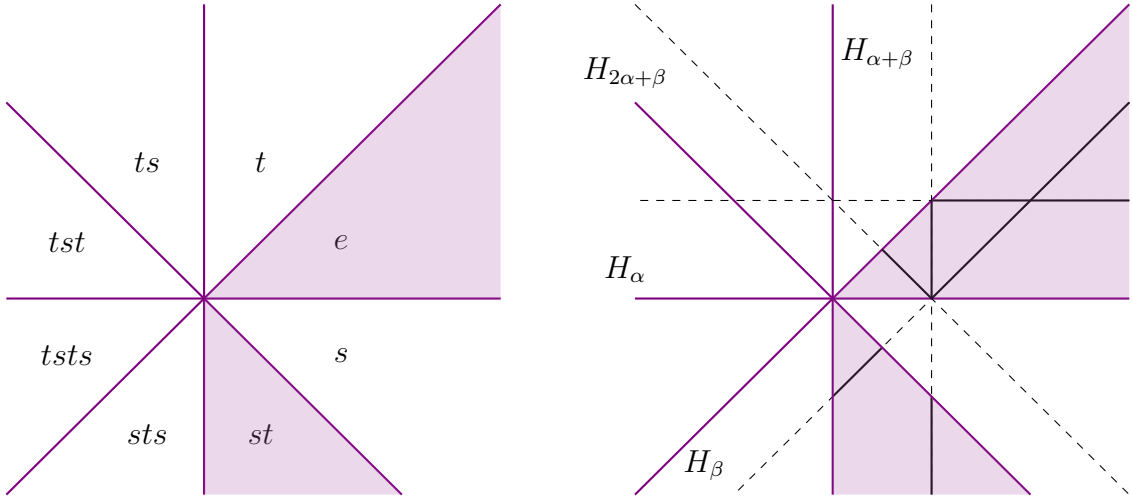
\begin{center}
 \begin{tikzpicture}[scale=1.3]
    \coordinate (a0+) at (3,0){};
    \coordinate (a0-) at (-2,0){};
    \coordinate (a1+) at (3,1){};
    \coordinate (a1-) at (-2,1){};

    \coordinate (b0+) at (3,3){};
    \coordinate (b0-) at (-2,-2){};
    \coordinate (b1+) at (3,2){};
    \coordinate (b1-) at (-1,-2){};

    \coordinate (ab0+) at (0,3){};
    \coordinate (ab0-) at (0,-2){};
    \coordinate (ab1+) at (1,3){};
    \coordinate (ab1-) at (1,-2){};

    \coordinate (2ab0+) at (-2,2){};
    \coordinate (2ab0-) at (2,-2){};
    \node[invisivertex] (2ab1+) at (-2,3){};
    \coordinate (2ab1-) at (3,-2){};

    \node[invisivertex] at (22.5:1.5cm){$e$};
    \node[invisivertex] at (67.5:1.5cm){$t$};
    \node[invisivertex] at (112.5:1.5cm){$ts$};
    \node[invisivertex] at (157.5:1.5cm){$tst$};
    \node[invisivertex] at (202.5:1.5cm){$tsts$};
    \node[invisivertex] at (247.5:1.5cm){$sts$};
    \node[invisivertex] at (292.5:1.5cm){$st$};
    \node[invisivertex] at (337.5:1.5cm){$s$};

    \path[-] (a0+) edge [color=violet, thick] node[above] {} (a0-);
    \path[-] (b0+) edge [color=violet, thick] node[above] {} (b0-);
    \path[-] (ab0+) edge [color=violet, thick] node[above] {} (ab0-);
    \path[-] (2ab0+) edge [color=violet, thick] node[above] {} (2ab0-);
    \fill[violet!60,nearly transparent] (0,0) -- (3,0) -- (3,3) -- cycle;
    \fill[violet!60,nearly transparent] (0,0) -- (0,-2) -- (2,-2) -- cycle;
  \end{tikzpicture}
  \qquad
  \begin{tikzpicture}[scale=1.3]
    \coordinate (a0+) at (3,0);
    \coordinate (a0-) at (-2,0);
    \coordinate (a1+) at (3,1);
    \coordinate (a1-) at (-2,1);
    \coordinate (a1++) at (3,1);
    \coordinate (a1--) at (1,1);

    \coordinate (b0+) at (3,3);
    \coordinate (b0-) at (-2,-2);
    \coordinate (b1+) at (3,2);
    \coordinate (b1-) at (-1,-2);
    \coordinate (b1++) at (3,2);
    \coordinate (b1--) at (1,0);
    \coordinate (b1+++) at (0.5,-0.5);
    \coordinate (b1---) at (0,-1);

    \coordinate (ab0+) at (0,3);
    \coordinate (ab0-) at (0,-2);
    \coordinate (ab1+) at (1,3);
    \coordinate (ab1-) at (1,-2);
    \coordinate (ab1++) at (1,1);
    \coordinate (ab1--) at (1,0);
    \coordinate (ab1+++) at (1,-1);
    \coordinate (ab1---) at (1,-2);

    \coordinate (2ab0+) at (-2,2);
    \coordinate (2ab0-) at (2,-2);
    \coordinate (2ab1+) at (-2,3);
    \coordinate (2ab1-) at (3,-2);
    \coordinate (2ab1++) at (0.5,0.5);
    \coordinate (2ab1--) at (1,0);

    \node[invisivertex] at (-1.3,-1.7){$H_\beta$};
    \node[invisivertex] at (-2.1,0.3){$H_\alpha$};
    \node[invisivertex] at (-2.1, 2.3){$H_{2\alpha+\beta}$};
    \node[invisivertex] at (0.45, 2.5){$H_{\alpha+\beta}$};

    \draw[-,color=violet, thick] (a0+) -- (a0-);
    \path[-] (a1+) edge [dashed] node[above] {} (a1-);
    \path[-] (a1++) edge [thick] node[above] {} (a1--);
    \path[-] (b0+) edge [color=violet, thick] node[above] {} (b0-);
    \path[-] (b1+) edge [dashed] node[above] {} (b1-);
    \path[-] (b1++) edge [thick] node[above] {} (b1--);
    \path[-] (b1+++) edge [thick] node[above] {} (b1---);
    \path[-] (ab0+) edge [color=violet, thick] node[above] {} (ab0-);
    \path[-] (ab1+) edge [dashed] node[above] {} (ab1-);
    \path[-] (ab1++) edge [thick] node[above] {} (ab1--);
    \path[-] (ab1+++) edge [thick] node[above] {} (ab1---);
    \path[-] (2ab0+) edge [color=violet, thick] node[above] {} (2ab0-);
    \path[-] (2ab1+) edge [dashed] node[above] {} (2ab1-);
    \path[-] (2ab1++) edge [thick] node[above] {} (2ab1--);
    \fill[violet!60,nearly transparent] (0,0) -- (3,0) -- (3,3) -- cycle;
    \fill[violet!60,nearly transparent] (0,0) -- (0,-2) -- (2,-2) -- cycle;
  \end{tikzpicture}
\end{center}
where we shaded the fundamental cone~$C$ and the cone~$stC$ and labelled all cones~$wC$ of the reflection arrangement by the elements $w \in W$.
The root poset~$\Phi^+$ (left) and its subposet on $E = \{\alpha,2\alpha+\beta\}$ (right) are
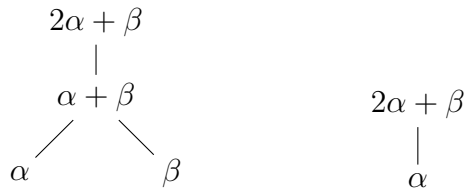
\begin{center}
  \begin{tikzpicture}[scale=1]
    \node[invisivertex] (alpha) at (0,0){$\alpha$};
    \node[invisivertex] (beta) at (2,0){$\beta$};
    \node[invisivertex] (alpha+beta) at (1,1){$\alpha + \beta$};
    \node[invisivertex] (2alpha+beta) at (1,2){$2\alpha + \beta$};

    \path[-] (alpha) edge [] node[above] {} (alpha+beta);
    \path[-] (beta) edge [] node[above] {} (alpha+beta);
    \path[-] (alpha+beta) edge [] node[above] {} (2alpha+beta);
  \end{tikzpicture}
  \hspace{2cm}
  \begin{tikzpicture}[scale=1]
    \node[invisivertex] (alpha) at (0,0){$\alpha$};
    \node[invisivertex] (2alpha+beta) at (0,1){$2\alpha + \beta$};
    \path[-] (alpha) edge [] node[above] {} (2alpha+beta);
  \end{tikzpicture}
\end{center}
We observe that~$E = \Phi^+\setminus\Inv(ts)$ by checking that $\Inv(ts) = \{\beta,\alpha+\beta\}$ is given by the roots corresponding to the two hyperplanes separating~$tsC$ from~$C$.
Next we illustrate the main results in these two settings, i.e., for the elements $e,st=(ts)^{-1} \in W$.
We have
\[
  \Anti_e = \big\{ \emptyset, \{\alpha\},\{\beta\},\{\alpha,\beta\},\{\alpha+\beta\},\{2\alpha+\beta\}\big\},\quad
  \Anti_{st} = \big\{\emptyset,\{\alpha\},\{2\alpha+\beta\}\big\}\,.
\]
\Cref{main-thm:chamber-bijection} now yields that the ceilings of the regions inside the cones~$C$ and~$stC$ are given by~$\one\Anti_\one = \Anti_\one$ and~$st\Anti_{st} = \big\{\emptyset,\{\beta\},\{\alpha+\beta\}\big\}$, respectively.
The cones~$C$ and~$stC$ are shaded above and one can check that the regions inside these cones have ceiling sets contained in~$\Anti_\one$ and~$st\Anti_{st}$, respectively.
Next, \Cref{main-thm:intersection-bijection} gives a bijection between $\Anti_w$ and $\LL_w$.
The intersection posets $\LL_\one$ and $\LL_{st}$ are drawn below, on the left and right respectively.
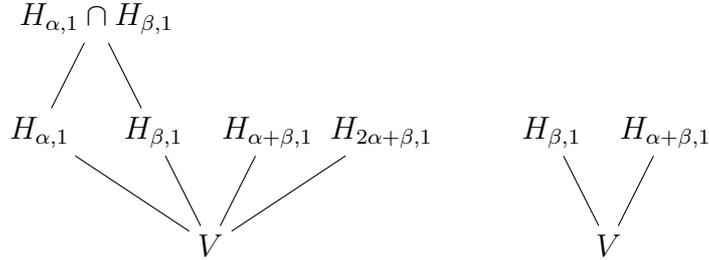
\begin{center}
  \begin{tikzpicture}[scale=1]
    \node[invisivertex] (V) at (2.25,0){$V$};
    \node[invisivertex] (alpha) at (0,1.5){$H_{\alpha,1}$};
    \node[invisivertex] (beta) at (1.5,1.5){$H_{\beta,1}$};
    \node[invisivertex] (alpha+beta) at (3,1.5){$H_{\alpha + \beta,1}$};
    \node[invisivertex] (2alpha+beta) at (4.5,1.5){$H_{2\alpha + \beta,1}$};
    \node[invisivertex] (X) at (.75,3){$H_{\alpha,1} \cap H_{\beta,1}$};

    \path[-] (V) edge [] node[above] {} (alpha);
    \path[-] (V) edge [] node[above] {} (beta);
    \path[-] (V) edge [] node[above] {} (alpha+beta);
    \path[-] (V) edge [] node[above] {} (2alpha+beta);
    \path[-] (alpha) edge [] node[above] {} (X);
    \path[-] (beta) edge [] node[above] {} (X);
  \end{tikzpicture}
  \qquad
  \begin{tikzpicture}[scale=1]
    \node[invisivertex] (V) at (.75,0){$V$};
    \node[invisivertex] (beta) at (0,1.5){$H_{\beta,1}$};
    \node[invisivertex] (alpha+beta) at (1.5,1.5){$H_{\alpha + \beta,1}$};
    \path[-] (V) edge [] node[above] {} (beta);
    \path[-] (V) edge [] node[above] {} (alpha+beta);
  \end{tikzpicture}
\end{center}

The six elements of $\LL_\one$ correspond to the antichains of $\Phi^+ = \Phi^+\setminus\Inv(\one)$, and the three elements of $\LL_{st}$ correspond, after applying the element~$st \in W$, to the antichains of $\Phi^+ \setminus \Inv(ts)$.
Their Poincar\'{e} polynomials are
\[
  \Poin(C,t) = 1 + 4t + t^2,\quad
  \Poin(stC,t) = 1 + 2t\,.
\]
In particular, the Whitney numbers of~$C$--that is, the coefficients of $\Poin(C,t)$--are precisely the type~$B_2$ Narayana numbers $(1,4,1)$.
We sum the Poincar\'e polynomials of all chambers to obtain
\begin{multline*}
  (1 + 4t + t^2) + (1 + 3t) + (1 + 2t) + (1 + t) + 1 + (1 + t) + (1 + 2t) + (1 + 3t) \\
  = 8 + 16t + t^2\, .
\end{multline*}
Evaluating at $t=1$ gives the total number $25 = 5^2$ of Shi regions.
Note that this is not the Poincar\'{e} polynomial of the full arrangement $\Shi(\Phi^+)$, which is $1 + 8t + 16t^2$.
In general, the constant term of the sum of Poincar\'{e} polynomials is the cardinality of the group (since each Poincar\'e polynomial for $\LL_w$ as constant term~$1$), whereas the Poincar\'{e} polynomial of the full arrangement always has constant term~$1$.

\section{Proofs of main results}
\label{sec:deleted}

In this section, we further generalize the main results to certain subarrangements of the Shi arrangement and consider the situation inside the dominant cone.
We then show in \Cref{sec:deduction} how these imply the results for other cones.

\subsection{Generalized results for Shi deletions}
\label{sec:deleted1}
Let $E \subseteq \Phi^+$ be a subposet and define the $E$-subarrangement
\[
  \Shi(E) = \Shi_0(\Phi^+) \cup \big\{ H_{\beta,1} \mid \beta \in E \big\} \subseteq \Shi(\Phi^+)\,.
\]
Shi deletions were considered by Armstrong and Rhoades under the name \emph{graphical Shi arrangements}~\cite{armstrong-rhoades}.
Let
\begin{itemize}
  \item $\RR_E$ be the regions of $\Shi(E)$ which lie inside the dominant cone $C$, i.e.,
  \[
  \big\{ R \in \RR\big(\Shi(E)\big) \mid R \subseteq C \big\},
  \]
  \item $\LL_E$ be the set of intersections of hyperplanes in~$\Shi(E)$ with nonempty intersection with $C$, i.e.,
  \[
  \big\{ X \in \LL\big(\Shi(E)\big) \mid X \cap C \neq \emptyset\big\}, \quad \text{and}
  \]
  \item $\Anti_E$ be the set of antichains of roots in~$E \subseteq \Phi^+$ (in root poset order).
\end{itemize}

Next we cast the main theorems into this framework.

\begin{thm}
\label{thm:1}
  The map
  \begin{align*}
    \varphi_E : \RR_E &\longrightarrow \Anti_E \\
    R &\longmapsto \ceil(R)
    \intertext{is a bijection with inverse}
    \varphi_E^{-1} : \Anti_E &\longrightarrow \RR_E \\
    A &\longmapsto \left\{
      v \in C ~\middle\vert~
      \begin{matrix}
      0 < \langle v, \gamma \rangle < 1 \text{ for } \gamma\in \ideal_E(A)\\[3pt]
      1 < \langle v, \gamma \rangle\hspace*{22pt} \text{ for } \gamma\notin \ideal_E(A)
      \end{matrix}
    \right\}\,,
  \end{align*}
  where
  \[\ideal_E(A) = \{ \gamma \in E \mid \gamma \prec \beta \text{ for some } \beta \in A\}\,.\]
\end{thm}

First observe that for $E = \Phi^+$, this map was given by Postnikov~\cite[Remark 2]{reiner97}.
We encode this in the following proposition, for later reference.

\begin{prop}[Postnikov~{\cite[Remark~2]{reiner97}}]
\label{prop:Postnikov}
  \Cref{thm:1} holds for $E = \Phi^+$.
\end{prop}

The more general case above can be deduced from this known case as follows.

\begin{clm}
\label{clm:1}
  The map~$\varphi_E: \RR_E \rightarrow \Anti_E$ is well-defined.
\end{clm}

\begin{proof}
  Let $R \in \RR_E$.
  We first observe that $\ceil(R)$ is a subset of~$E$ by construction.
  Thus we only need to show that $\ceil(R)$ is an antichain in the root poset.
  To this end, let $\beta \in \ceil(R)$ and let $\gamma \in \Phi^+$ such that $\gamma \prec \beta$.
  Then $\beta - \gamma = \sum_{\alpha\in\Delta} c_\alpha \alpha$ where $c_\alpha \geq 0$ are integers and not all $c_\alpha = 0$.
  Then for all $v\in \overline{R} \cap C$, we have
  \begin{align*}
    \langle \gamma, v\rangle
     & = \langle  \gamma - \beta + \beta, v\rangle
     = \langle  \beta, v\rangle - \langle  \beta - \gamma, v\rangle
    \leq ~ 1 - \sum_{\alpha\in\Delta} c_\alpha ~ \langle \alpha, v\rangle < 1,
  \end{align*}
  where the final inequality uses that $\langle \beta,v\rangle \leq 1$ since $R \subset H_{\beta,1}^-$ because $\beta \in \ceil(R)$ and also that $\langle\alpha,v\rangle > 0$ for all~$\alpha\in\Delta$, since $v \in \overline{R} \subseteq C$.
  The inequality $\langle\gamma,v\rangle < 1$ is strict, preventing the hyperplane $H_{\gamma,1}$ to be facet-defining for the region~$R$, $\gamma \notin \ceil(R)$.
\end{proof}

\begin{prop}
\label{prop:identifyregion}
  Let $R \in \RR_E$ and $A = \ceil(R)$.
  Let $R' = \varphi_{\Phi^+}^{-1}(A) \in \RR_{\Phi^+}$ be the unique dominant region of the Shi arrangement $\Shi(\Phi^+)$ with ceiling set~$A$.
  Then $R' \subseteq R$.
\end{prop}

\begin{proof}
\changes{  First note that $A = \ceil(R)$ is an antichain in~$E$ and thus also an antichain in~$\Phi^+$.
\Cref{prop:Postnikov} implies that the unique dominant region~$R' = \varphi_{\Phi^+}^{-1}(A)$ exists.
}

  Let now $H_{\beta_1,1},\dots,H_{\beta_k,1},H_{\gamma_1,1},\dots,H_{\gamma_\ell,1},H_{\alpha_1,0},\dots,H_{\alpha_m,0}$ be the facet-defining hyperplanes of the region~$R$ for $\beta_1,\dots,\beta_k,\gamma_1,\dots,\gamma_\ell \in E$ and $\alpha_1,\dots,\alpha_m \in \Delta$, so that
  \begin{equation}
  \label{eq:Rfactets}
  \tag{$\star$}
    R = H_{\beta_1,1}^- \cap \dots \cap H_{\beta_k,1}^-\quad \cap\quad H_{\gamma_1,1}^+ \cap \dots \cap H_{\gamma_\ell,1}^+\quad \cap\quad H_{\alpha_1,0}^+ \cap \dots \cap H_{\alpha_m,0}^+\,.
  \end{equation}
  We aim to show that the dominant Shi region~$R'$ is on the same side for each of these hyperplanes to conclude $R' \subseteq R$.
 Note that we chose $R'$ so that $\ceil(R) = \ceil(R') = \{\beta_1,\dots,\beta_k\}$.
 Thus
  \[
    R' \subseteq H_{\beta_1,1}^- \cap \dots \cap H_{\beta_k,1}^-\,.
  \]
  Since $R'\subseteq C \subseteq H_{\alpha_1,0}^+ \cap \dots \cap H_{\alpha_m,0}^+$, it remains to show that
  \begin{equation}
  \label{eq:RRprime}
  \tag{$\star\star$}
    R' \subseteq H_{\gamma_1,1}^+ \cap \dots \cap H_{\gamma_\ell,1}^+\,.
  \end{equation}
  The proof of \Cref{clm:1} implies for $\gamma \in \Phi^+$ that
  \begin{align*}
    \gamma \in \ideal_{\Phi^+}(A) &\Longrightarrow 
    R \subseteq H_{\gamma,1}^-\,.
  \intertext{
  We in particular obtain for $\gamma \in \Phi^+$ that
  }
    R \subseteq H_{\gamma,1}^+ &\Longrightarrow \gamma \in \Phi^+ \setminus \ideal_{\Phi^+}(A)\,.
  \intertext{
  Since we have $R \subseteq H_{\gamma_1,1}^+ \cap \dots \cap H_{\gamma_\ell,1}^+$ from~\eqref{eq:Rfactets}, we obtain $\gamma_1,\dots,\gamma_\ell \in \Phi^+ \setminus \ideal_{\Phi^+}(A)$.
  Since the map $\varphi_{\Phi^+}^{-1}$ is well-defined and bijective by \Cref{prop:Postnikov} and we obtain
  }
    \gamma \in \Phi^+ \setminus \ideal_{\Phi^+}(A) &\Longrightarrow R' \subseteq H_{\gamma,1}^+\,.
  \end{align*}
  This establishes~\eqref{eq:RRprime} and we conclude that $R' \subseteq R$.
\end{proof}

\begin{clm}
  \label{clm:phiinjective}
  The map~$\varphi_E: \RR_E \rightarrow \Anti_E$ is injective.
\end{clm}

\begin{proof}
  \changes{Let $R_1,R_2 \in \RR_E$ such that $\ceil(R_1) = \ceil(R_2)$ and set \[S = \varphi_{\Phi^+}(\ceil(R_1)) = \varphi_{\Phi^+}(\ceil(R_2))\,.\]
  Note that $S$ is nonempty.
  From \Cref{prop:identifyregion}, we have $S \subseteq R_1$ and $S \subseteq R_2$.
  In particular, $S\subseteq R_1 \cap R_2$ and thus $R_1 \cap R_2$ is nonempty.
  Since $R_1$ and $R_2$ are regions of the same hyperplane arrangement and have nonempty intersection, they must coincide.}
\end{proof}

\begin{clm}
\label{clm:phiinvwelldefined}
  The map~$\varphi_E^{-1} : \Anti_E \rightarrow \RR_E$ is well-defined.
\end{clm}

\begin{proof}
  Let~$A \in \Anti_E$.
  Then also $A \in \Anti_{\Phi^+}$.
  By \Cref{prop:Postnikov}, there exists a point~$x \in C$ satisfying all inequalities for $\varphi_{\Phi^+}^{-1}(A)$ given in the definition in \Cref{thm:1}.
  This point~$x$ thus in particular satisfies all inequalities for $\varphi_{E}^{-1}(A)$ as these are a subset of the inequalities for $\varphi_{\Phi^+}^{-1}(A)$.
  These inequalities thus define a region of the arrangement~$\Shi(E)$.
\end{proof}

\begin{clm}
  \label{clm:phiinvinjective}
  The map~$\varphi_E^{-1} : \Anti_E \rightarrow \RR_E$ is injective.
\end{clm}

\begin{proof}
  Let $A,A' \in \Anti_E$ such that $\varphi_E^{-1}(A) = \varphi_E^{-1}(A')$.
  Denoting this region by~$R$, the definition of $\varphi_{E}^{-1}$ in \Cref{thm:1} yields for $\beta \in E$ that
  \[
    \beta \in \ideal_{E}(A)
    \Longleftrightarrow R \subseteq H_{\beta,1}^-
    \Longleftrightarrow \beta \in \ideal_{E}(A')\,.
  \]
  The first equivalence is given by the inequalities for $\varphi_E^{-1}(A)$ and the second by the inequalities for $\varphi_E^{-1}(A')$.
  This means that $\ideal_{E}(A) = \ideal_{E}(A')$ and we conclude~$A = A'$.
\end{proof}

\begin{proof}[Proof of \Cref{thm:1}]
  From \Cref{clm:1} and \Cref{clm:phiinvwelldefined}, both $\varphi_E$ and $\varphi_{E}^{-1}$ are well-defined.
  From \Cref{clm:phiinjective} and \Cref{clm:phiinvinjective}, both maps are injective.
  \changes{Now we have injective maps in both directions.
  Since all sets are finite, $\varphi_E$ and $\varphi_E^{-1}$ are both bijections.}
  
  It remains to show that they are inverses of one another, \changes{i.e., we need to show that our notation $\varphi_E$ and $\varphi_E^{-1}$ is justified}.
  To see this, let~$R \in \RR_E$ and $A = \varphi_E(R) \in \Anti_E$.
  Consider the unique dominant Shi region~$R' \in \RR_{\Phi^+}$ with $\ceil(R') = \ceil(R) = A$.
  The inequalities for $\varphi_{\Phi^+}^{-1}(A)$ contain the inequalities for $\varphi_{E}^{-1}(A)$, implying
  \[
    R' = \varphi_{\Phi^+}^{-1}(A) \subseteq \varphi_{E}^{-1}(A)\,.
  \]
  On the other hand, \Cref{prop:identifyregion} implies that $R' \subseteq R$.
  \changes{Since $R' \subseteq \varphi_{E}^{-1}(A)$ and $R' \subseteq R$, we have $R' \subseteq \varphi_{E}^{-1}(A) \cap R$.
  Since $\varphi_{E}^{-1}(A)$ and~$R$ are both regions of the same hyperplane arrangement and have nonempty intersection, we obtain $\varphi_{E}^{-1}(A) = R$, as desired.}
\end{proof}

\begin{thm}
\label{thm:2}
  The map
  \begin{align*}
    \psi_E : \LL_E &\longrightarrow  \Anti_E \\
    X &\longmapsto \big\{\beta \in E \mid X \subseteq H_{\beta,1}\big\}
    \intertext{is a bijection with inverse}
    \psi_E^{-1} : \Anti_E &\longrightarrow \LL_E \\
    A &\longmapsto \bigcap_{\beta\in A}H_{\beta,1}\,.
  \end{align*}
\end{thm}

Before proving this theorem, we provide the following well-known property for later reference.

\begin{prop}
\label{prop:geometric}
  Every lower interval in $\LL_E$ is a geometric lattice.
  In particular, $\LL_E$ is atomic in the sense that every element $X \in \LL_E$ is the join of the atoms in $\LL_E$ below~$X$.
\end{prop}

\begin{proof}
  Every interval in the intersection poset of a hyperplane arrangement is geometric, see e.g.~\cite[Proposition~3.8]{stanley-hyperplanes} or~\cite{wachs-walker}.
  As every lower interval in~$\LL_E$ is also a lower interval in the intersection poset of the hyperplane arrangement~$\Shi(E)$, the statement follows.
\end{proof}

We also recall the following theorem of Sommers.

\begin{prop}[\!\!\cite{sommers}]
\label{prop:sommers}
  Let $A \subseteq \Phi^+$ be an antichain.
  Then there exists an element $w \in W$ in the Weyl group such that $w(A) \subseteq \Delta$.
  In particular, $A$ is linearly independent.
\end{prop}

\begin{clm}
\label{clm:2}
  The map~$\psi_E: \LL_E \rightarrow \Anti_E$ is well-defined.
\end{clm}

\begin{proof}
  Let $X \in \LL_E$.
  This means there exists $B \subseteq E$ such that
  \[
    X = \bigcap_{\beta\in B} H_{\beta,1} \cap C \neq \emptyset\,.
  \]
  Without loss of generality, we may assume that $B = \{ \beta \in E \mid X \subseteq H_{\beta,1}\}$.
  We need to show that $B \subset E$ is an antichain in the root poset.
  For the sake of contradiction, assume that~$B$ is not an antichain, meaning that there exist $\gamma \prec \beta$ in~$B$.

  Let $v \in X$.
  Since $v \in H_{\gamma,1} \cap H_{\beta,1}$, we have $\langle v,\gamma \rangle = \langle v,\beta \rangle = 1$ and thus
  \[
    \langle v, \beta-\gamma \rangle = 0\,.
  \]
  Because $\beta - \gamma = \sum_{\alpha\in\Delta} c_\alpha \alpha$ with $c_\alpha \geq 0$ and not all $c_\alpha = 0$, we obtain that there exists~$\alpha \in \Delta$ with $\langle v, \alpha\rangle \leq 0$.
  This is a contradiction to the fact that $v \in X \subseteq C$ and thus, $\langle v,\alpha\rangle > 0$ for all $\alpha \in \Delta$.
\end{proof}

\begin{clm}
  The map~$\psi_E: \LL_E \rightarrow \Anti_E$ is injective.
\end{clm}

\begin{proof}
  From \Cref{prop:geometric}, $X$ is the join of all the atoms below~$X$ in $\LL_E$.
  Geometrically: the atoms below $X$ are the hyperplanes containing~$X$.
  From \Cref{clm:2}, the defining-roots of these hyperplanes form an antichain of the root poset.
  This uniquely describes~$X$, so~$\psi_E$ is injective.
\end{proof}

\begin{clm}
\label{clm:3}
  We have that the interval $[V,X]$ is Boolean for all $X\in \LL_E$.
  In particular, $\mu(V,X) = (-1)^{\codim(X)}$ and
  \[
    \#\LL_E = \#\Anti_E\,.
  \]
\end{clm}

\begin{proof}
  Certainly $X = \bigcap_{\beta\in A} H_{\beta,1}$ for $A \subseteq E$ and we have seen in the proof of \Cref{clm:2} that~$A$ is an antichain in~$E$ and also in~$\Phi^+$.
  It thus follows from \Cref{prop:sommers} that~$A$ is linearly independent and $[V,X]$ is a Boolean lattice.
  This implies $\mu(V,X) = (-1)^{\codim(X)}$ and Zaslavsky's theorem~\cite[Example A]{zaslavsky77} gives
  \[
    \# \RR_E = \sum_{X\in \LL_E} |\mu(V,X)| = \# \LL_E. \qedhere
  \]
\end{proof}

\begin{note}
  Since each lower interval $[V,X]$ is Boolean, the previous proof also shows that
  \[
    \mu(X,Y) = (-1)^{\codim(X) + \codim(Y)}
  \]
  for all $X, Y \in \LL_E$ with $Y\subseteq X$.
  In particular, this means that $\LL_E$ is Eulerian.
\end{note}

\Cref{clm:3} has the following immediate consequence for the \Dfn{Poincar\'e polynomial}
\[
  \Poin_E(C,t) = \sum_{X \in \LL_E} t^{\codim(X)} = \sum_{k \geq 0} \whit{E,k}{C}\cdot t^k
\]
of the dominant cone~$C$ in the Shi deletion~$\Shi(E)$ and where $\whit{E,k}C$ denotes its \Dfn{$k$\th Whitney number}.

\begin{coro}
\label{prop:whitney-numbers}
  Let $E \subseteq \Phi^+$.
  Then the $k$\th Whitney number of $\LL_E$ is given by
  \[
    \whit{E,k}{C} = \# \{ A\in\Anti_E \mid \# A = k \}\,.
  \]
\end{coro}

\begin{proof}[Proof of \Cref{thm:2}]
  The preceeding sequence of claims shows that $\psi_E : \LL_E \to \Anti_E$ is well-defined and injective.
  In \Cref{thm:1}, we saw that $\# \Anti_E = \# \RR_E$.
  With \Cref{clm:3}, we obtain $\# \LL_E = \# \Anti_E$ and thus, $\psi_E$ is a bijection.

  It remains to show that $\psi_E^{-1}$ is well-defined and the inverse of~$\psi_E$.
  We prove both simultaneously.
  Let~$A \in \Anti_E$.
  Since $\psi_E$ is a bijection, there is a unique $X \in \LL_E$ with $\psi_E(X) = A$.
  This means that $X = \bigcap_{\beta \in A} H_{\beta,1}$.
  This implies that
  \[
    X = \psi_E^{-1}(A) = \psi_E^{-1}\circ \psi_E(X)\,.
  \]
  Since this holds for any $A \in \Anti_E$, we have that~$\psi_E^{-1}$ is well-defined and the inverse of~$\psi_E$.
\end{proof}

\subsection{Deduction of main results}
\label{sec:deduction}

In this section we deduce the main results of this paper from the results in the previous section.
When $E = \Phi^+ \setminus \Inv(w^{-1})$, we show that $\Anti_E = \Anti_w$, $\whit{E,k}{C} = \whit{k}{wC}$, and\footnote{Explicitly including the additional equalities for the Poincaré polynomials was suggested to us by one of the referees. We happily include them and take this opportunity to again thank the referees for their extremely detailed analysis of our arguments and possible implications!}
\[
  \Poin_E(C,t) = \Poin(wC,t)\,.
\]
The main ingredient is the following lemma of Armstrong-Reiner-Rhoades (we include the proof for completeness, as the proof is short and we use a slightly reframed version of their statement).
Recall that the inversion set of an element $w\in W$ is
\[
\Inv(w) = \Phi^+ \cap w\Phi^- = \{ \beta \in \Phi^+ \mid w^{-1}(\beta) \in \Phi^-\}\,.
\]

\begin{lem}[\!\!{\cite[Lemma 10.2]{Armstrong-Reiner-Rhoades}}]
\label{lem:inversions}
  Let $w\in W$ and $\beta \in\Phi^+$.
  Then $H_{\beta,1}$ cuts through $wC$ if and only if $\beta$ is not an inversion of $w$.
  In symbols,
  \[
    H_{\beta,1} \cap wC = \emptyset \iff \beta \in \Inv(w)\,.
  \]
\end{lem}

\begin{proof}
  We have
  \begin{align*}
    H_{\beta,1} \cap wC \neq \emptyset &\iff \text{there exists } v \in C \text{ such that } \langle \beta,w(v) \rangle = 1\\
    &\iff \text{there exists } v \in C \text{ such that } \langle w^{-1}(\beta),v \rangle = 1\\
    &\iff w^{-1}(\beta)\in\Phi^+ \\
    &\iff \beta \notin \Inv(w).
  \end{align*}
  Here, the first and last equivalences are the respective definitions, the second uses the $W$-invariance of the inner product, and the third uses the fact that the dominant cone can be described by $C = \big\{ v \in V \mid \langle\gamma,v\rangle > 0 \text{ for all } \gamma \in \Phi^+\big\}$.
\end{proof}

The following lemma is well-known and appears in several places including~\cite[Section 4.1]{dyer}.
We include its proof for completeness.

\begin{lem}
\label{lem:inversionset}
  For $w \in W$, we have that $\Phi^+\setminus\Inv(w) = w\big(\Phi^+\setminus\Inv(w^{-1})\big)$.
\end{lem}

\begin{proof}
  The following is a straightforward calculation.
  \begin{align*}
    w^{-1}\big(\Phi^+ \setminus \Inv(w)\big)
    & = w^{-1}(\Phi^+) \setminus w^{-1}(\Phi^+ \cap w\Phi^-)\\
    & = w^{-1}(\Phi^+) \setminus \big(w^{-1}\Phi^+ \cap \Phi^-\big)\\
    & = w^{-1}\Phi^+ \cap \Phi^+\\
    & = \Phi^+ \setminus \Inv(w^{-1}).
  \end{align*}
  Applying now~$w$ to both sides yields the statement.
\end{proof}

\begin{proof}[Proof of \Cref{main-thm:chamber-bijection,,main-thm:intersection-bijection,,main-thm:set-cardinality-equality}]
  These follow from the respective theorems in \Cref{sec:deleted} applied to the set
  \[
    E = \Phi^+ \setminus \Inv(w^{-1})\,.
  \]
  We first observe that $\Anti_w = \Anti_E$.
  Second, we show $\RR_w = w(\RR_E)$.
  To this end, observe that \Cref{lem:inversions} implies that the hyperplanes cutting through~$wC$ are $\big\{ H_{\beta,1} \mid \beta \in \Phi^+ \setminus \Inv(w) \big\}$.
  \Cref{lem:inversionset} then implies that this set equals 
  \[
    \big\{ H_{w(\beta),1} \mid \beta \in E \big\}
    = \big\{ w(H_{\beta,1}) \mid \beta \in E \big\}\,.
  \]
  Therefore,~$w$ is a bijection between hyperplanes of $\Shi(E)$ inside~$C$ and hyperplanes of $\Shi(\Phi^+)$ inside~$w(C)$.
  This implies $\RR_w = w(\RR_E)$, as desired.
  Since the intersection poset is invariant under linear isomorphisms, we also obtain the poset isomorphism $\LL_w \cong w(\LL_E)$.
\end{proof}

\section{Algebraic interpretation}
\label{sec:algebra}

Fix a subset $E \subseteq \Phi^+$ of positive roots throughout this section.
We present two equivalent viewpoints on an algebraic interpretations of the Poincar\'e polynomial
\[
  \Poin_E(C,t) = \sum_{k \geq 0} \whit{E,k}{C}\cdot t^k
\]
from \Cref{prop:whitney-numbers} as the Hilbert series of a graded ring, see \Cref{coro:Zisos}.
This ring is a special case of a more general object called the \emph{Varchenko-Gel'fand ring} of a cone.
For a given Weyl cone $wC$ of a Shi arrangement (or of a Shi deletion), we define a Varchenko-Gel'fand ring for that cone and use that to obtain $\Poin_E(C,t)$.
While not, \emph{a priori}, a graded ring, the Varchenko-Gel'fand ring admits a filtration and a general result of the first author shows that the Hilbert series of the associated graded is precisely the Poincaré polynomial of $wC$; see~\cite[Theorem 1]{dorpalenbarry}.

Independently, Chapoton gave a presentation for the Varchenko-Gel'fand ring of the dominant cone of a Shi arrangement via the combinatorics of the root poset in~\cite{chapoton}.
Using the work of Armstrong-Reiner-Rhoades~\cite{Armstrong-Reiner-Rhoades}, we extend Chapoton's argument to all Weyl cones, and obtain a presentation of the Varchenko-Gel'fand ring of a Weyl cone of a Shi arrangement via antichains of the root poset.
Comparing bases of the two presentations of the Varchenko-Gel'fand ring gives the result.

While the Varchenko-Gel'fand ring is implicitly geometric (or, more specifically, defined using the language of oriented matroids), Chapoton's presentation depends only on the poset structure of the root poset.
This motivates us to define a ring which we call the \emph{order ring} of an arbitrary poset $P$.
This ring coincides with the Varchenko-Gel'fand ring of a Weyl cone when $P$ is a (subposet of) a root poset, and (in general) we find that it coincides with the \emph{coordinate ring} of the vertices of the \emph{order polytope} of $P$.

\subsection{Two equivalent ring constructions}

We present the following two $\C$-algebras associated to the subset~$E \subseteq \Phi^+$.
The first can be associated to any cone in an arrangement~\cite{VarchenkoGelfand,dorpalenbarry}, while the second can be associated to any finite poset~\cite{chapoton}.
To avoid introducing a lot of bulky notation, we define the Varchenko-Gel'fand ring only for cones of Shi deletions. 
For a more general treatment, see~\cite{dorpalenbarry}.

\medskip

The \Dfn{Varchenko-Gel'fand ring} of the dominant cone~$C$ of the Shi deletion $\Shi(E)$ is
\[
  \VG(E,C) = \big\{ f:\RR_E \to \C \big\}
\]
with pointwise addition and multiplication.
It is linearly generated by the \Dfn{indicator functions}
\[
  \big\{ \delta_R : \RR_E \to \C \mid R \in \RR_E\big\}\,.
\]
It is not hard to see (and is carefully discussed in~\cite{VarchenkoGelfand} and in~\cite[Section 2.3]{dorpalenbarry}) that $\VG(E,C)$ is also generated (as a ring) by the \Dfn{Heaviside functions} $\big\{x_\beta: \RR_E \to \C \mid \beta \in E \big\}$ with
\[
  R \mapsto
  \begin{cases}
    1 & \text{ if } R\subseteq H_{\beta,1}^-\\
    0 & \text{ else}.
  \end{cases}
\]
To be explicit, we have
\[
  x_\beta = \sum_{R\subseteq H_{\beta,1}^-} \delta_R,\quad \delta_R = \prod_{R\subseteq H_{\beta,1}^-} x_\beta ~ \prod_{R\subseteq H_{\beta,1}^+} (1 - x_\beta).
\]
In \Cref{thm:1}, we saw for a region $R \in \RR_E$ with ceiling set $A = \ceil(R) \in \Anti_E$ and a root $\beta \in E$, that
\begin{equation}
  x_\beta(R) = 1 \iff R \subseteq H_{\beta,1}^- \iff \beta \in \ideal_E(A)\,.
  \tag{$\star\star$}
  \label{eq:VGring}
\end{equation}

This means that the Varchenko-Gel'fand ring is in the present case isomorphic to the following ring construction for any finite poset.

\medskip

Following a construction by Chapoton~\cite{chapoton}, we define the \Dfn{order ring} of the subposet~$E \subseteq \Phi^+$.
To this end, let $\ideals_E = \big\{ \ideal_E(A) \mid A \in \Anti_E \big\}$ be its sets of order ideals inside $E \subseteq \Phi^+$ and define the $\C$-algebra
\[
  \FP(E) = \big\{ f : \ideals_E \to \C \big\}
\]
with pointwise addition and multiplication.
It is generated linearly (as a $\C$-module) by the \Dfn{indicator functions}
\[
  \big\{ \delta_I : \ideals_E \to \C \mid I \in \ideals_E\big\}
\]
and it is also generated (as a ring) by the \Dfn{Heaviside functions} $\big\{y_\beta: \ideals_E \to \C \mid \beta \in E \big\}$ given by
\[
  I \mapsto
  \begin{cases}
    1 & \text{ if } \beta\in I\\
    0 & \text{ else}.
  \end{cases}
\]
Concretely, we have
\[
  y_\beta = \sum_{I\in\ideals_E\,:\,\beta \in I} \delta_I,\quad \delta_I = \prod_{\beta \in I} y_\beta ~\cdot \prod_{\beta\in E\setminus I} (1 - y_\beta).
\]
These Heaviside functions allow us to express $\FP(E)$ as the quotient of a polynomial ring.
Define a surjective ring homomorphism
\[
  f: \C[z_\beta \mid \beta \in E] \twoheadrightarrow \FP(E)
\]
by sending~$1$ to~$1$, $z_\beta$ to~$y_\beta$.
We obtain
\[
  \FP(E) \cong \C[z_\beta \mid \beta \in E]\big/\ker(f)\,.
\]
Note that $\beta\preceq \gamma$ in $E$ implies $\ideal(\beta) \subseteq \ideal(\gamma)$.
We thus obtain $z_\beta(1 - z_\gamma) \in \ker(f)$.
We aim to show that these relations generate~$\ker(f)$.
We start with collecting the following auxiliary results.
The first is an observation that follows from~\Cref{eq:VGring}.

\begin{prop}
\label{prop:vg-or-bijection}
  Let $R \in \RR_E$ and let $I = \ideal_E(A)$ be the order ideal in~$E$ generated by $A = \ceil(R)$.
  Then
  \[
  \begin{array}{ccc}
    \VG(E,C) & \cong  &\FP(E) \\
    \delta_R & \longmapsto     &\delta_I
  \end{array}
  \]
  is a ring isomorphism between~$\VG(E,C)$ and~$\FP(E)$.
  In particular, they are both free~$\C$-modules of the same dimension.
\end{prop}

The following is standard in the theory of Gröbner bases, and can be found in~\cite[Lemma~7, Remark~8]{dorpalenbarry}.
In what follows, we will use the notation from~\cite[Section~2]{dorpalenbarry} regarding filtrations, rings, and associated graded rings.
For a more general treatment of Gröbner bases and in particular for definitions of initial ideals and standard monomials, we refer to~\cite[Chapters 1-5]{cox-little-oshea} and to~\cite[Chapter~1]{Sturmfels}.

\begin{lem}
\label{lem:integer-grobner}
  Let $\omega$ be a monomial ordering and $I$ be some finite indexing set.
  Assume one has a $\C$-algebra surjection
  $f : \C[z_1,\ldots,z_n] \twoheadrightarrow R$
  in which $R$ is a free $\C$-module of dimension $r$,
  and $\GGG=\{g_i\}_{i \in I} \subset \C[z_1,\ldots,z_n]$ has the following properties:
  \begin{itemize}
  \item[(i)]  $\GGG \subset \ker(f)$.
  \item[(ii)] The set of $\init_{\omega}\GGG$-standard monomials $\Anti=\{m_1,\dots,m_t\}$ has cardinality $t\leq r$.
  \end{itemize}
  Then $\mathcal{G}$ is a Gröbner basis for $\ker(f)$.
  If $\omega$ is also a degree-ordering, then $\init_{deg}\mathcal{G}$ is a Gröbner basis for the corresponding ideal of the associated graded \changes{ring} with respect to the degree filtration.
\end{lem}

\begin{prop}
\label{prop:fp-vg-generating-set}
  We have
  \[
  \FP(E) \cong \C[z_\beta \mid \beta \in E] \big/ \langle z_\beta(1 - z_\gamma) \mid \beta \preceq \gamma\rangle \,.
  \]
\end{prop}

\begin{proof}
  Let $f: \C[e_\beta \mid \beta \in E] \to \FP(E)$ as above~\Cref{lem:integer-grobner} so that we have
  \[
    \FP(E) \cong S/\ker(f)\,.
  \]
  We have already seen that
  \[
    \GGG = \{ z_\beta(1 - z_\gamma) \mid \beta \preceq \gamma \} \subseteq \ker(f)
  \]
  For any monomial order~$\omega$, the initial term of $e_\beta(1 - e_\gamma)$ is $-e_{\beta}e_{\gamma}$.
  Since~$\GGG$ contains one relation $e_{\beta}(1 - e_\gamma)$ for each order relation of the poset~$E$, the initial terms of elements of~$\GGG$ are
  \[
    \init_{\omega}\GGG = \{ -z_{\beta}z_{\gamma} \mid \beta \preceq \gamma \text{ in } E \}\,.
  \]
  Thus the standard monomials are given by $\{ \prod_{\beta \in A} z_\beta \mid A \in \Anti_E\}$.
  \Cref{thm:1} now implies that the number of standard monomials equals the number of regions in~$\RR_E$.
  It follows from~\cite[Example~A]{zaslavsky77} and~\cite[Theorem~1]{dorpalenbarry} that this number of regions equals the dimension of~$\VG(E,C)$.
  By~\Cref{prop:vg-or-bijection}, $\FP(E)$ and $\VG(E,C)$ are free of the same dimension.
  With~\Cref{lem:integer-grobner}, the statement follows.
\end{proof}

\begin{coro}
\label{coro:Zisos}
  For any term order, we have a $\C$-algebra isomorphism
  \begin{align*}
    \mathfrak{gr}\big(\FP(E)\big) & \cong \C[z_\beta \mid \beta \in A]\,\big/\, \langle z_\beta z_\gamma \mid \beta \preceq \gamma \rangle
  \end{align*}
  where $\mathfrak{gr}\big(\FP(E)\big)$ is the associated graded ring with respect to the degree filtration.
  Moreover,
  \[
    \operatorname{Hilb}\left(\mathfrak{gr}\big(\FP(E)\big),t\right) = \sum_{A\in\Anti_E} t^{\# A} = \Poin_E(C,t)\,.
  \]
\end{coro}

\begin{proof}
  The first part is an immediate consequence of the proof of \Cref{prop:vg-or-bijection} using \Cref{lem:integer-grobner}.
  To obtain the first equality for the Hilbert series, observe that the $k$\th\ graded component is linearly generated by the squarefree monomials
    \[
      \left\{ \prod_{\beta \in A} z_\beta \mid A \in \Anti_E \text{ with } |A| = k \right\}\,.
    \]
    The second equality finally follows with\cite[Theorem~1]{dorpalenbarry} using $\VG(E,C) \cong  \FP(E)$ from \Cref{prop:vg-or-bijection}.
\end{proof}

\begin{ex}
  Consider the permutation $w = 24351$ in the symmetric group on $\{1,\dots,5\}$.
  The non-inversions of its inverse $w^{-1} = 51324$ are 
  \[
    E = \{ e_2 - e_3,\ e_2 - e_4,\ e_2 - e_5,\ e_3 - e_5,\ e_4 - e_5\} \,.
  \]
  Below is a picture of the root poset for the root system of type~$A_4$, with the non-inversions shaded:
  \begin{center}
  \begin{tikzpicture}[scale=0.8]
  \foreach \x in {0,...,3} {
  \foreach \y in {0,...,\x} {
    \fill (\x - .5*\y,\y) circle (0.1);
  }
  }
  \draw[] (0,0) -- (.5,1) -- (1,0) -- (1.5,1) -- (2,0) -- (2.5,1) -- (3,0);
  \draw[] (.5,1) -- (1,2) -- (1.5,1) -- (2,2) -- (2.5,1);
  \draw[] (1,2) -- (1.5,3) -- (2,2);
  
  \fill[violet!50] (1,0) circle (0.15);
  \fill[violet!50] (1.5,1) circle (0.15);
  \fill[violet!50] (2.5,1) circle (0.15);
  \fill[violet!50] (2,2) circle (0.15);
  \fill[violet!50] (3,0) circle (0.15);
  \end{tikzpicture}
  \end{center}
  In this case, we have
  \begin{multline*}
    \Anti_E = \big\{ \emptyset,\ \{e_2 - e_3\},\ \{e_2 - e_4\},\ \{e_2 - e_5\},\ \{e_3 - e_5\},\ \{e_4 - e_5\},\\
    \{e_2 - e_3, e_3-e_5\},\ \{e_2 - e_4, e_4-e_5\},\ \{e_2-e_3, e_4 - e_5\},\ \{e_2-e_4, e_3 - e_5\} \big\}
  \end{multline*}
  and we obtain $\Poin_E(C,t) = 1 + 5t + 4t^2$.
  Indexing the variables according to the non-inversions, we obtain that $\FP(E)$ is a quotient of $\C[z_{23}, z_{24}, z_{25}, z_{35}, z_{45}]$ by the ideal as given in \Cref{thm:C-iso}.
  The defining ideal of the associated graded is generated by the relations
  \[
    \{z_{23}^2, z_{24}^2, z_{25}^2, z_{35}^2, z_{45}^2, \quad z_{23}z_{24}, z_{23}z_{25}, z_{24}z_{25}, z_{45}z_{35}, z_{45}z_{25}, z_{35}z_{25}\}\,.
  \]
  From this, we can compute the standard monomials
  \[
    \big\{1,\quad
    z_{23}, z_{24}, z_{25}, z_{35}, z_{45},\quad
    z_{12}z_{34}, z_{12}z_{25}, z_{24}z_{34}, z_{24}z_{35}, z_{23}z_{25}, z_{25}z_{34}\}\,,
  \]
  and see that
  \[
    \operatorname{Hilb}\left(\mathfrak{gr}\big(\FP(E)\big),t\right) = 1 + 5t + 4t^2 = \Poin_E(C,t)\,.
  \]
\end{ex}

\subsection{Order rings and order polytopes}
\label{sec:order-polys}

In the section, we motivate our choice of name by extending the order ring to general finite posets and discuss its connection to the order polytope.
Based on the calculation to obtain \Cref{prop:fp-vg-generating-set}, we show that the order ring is the \emph{coordinate ring} of the vertices of the order polytope.

\medskip
For a finite poset~$P = (P,\preceq)$, denote the set of antichains by $\Anti_P$ and extend the definition of its \Dfn{order ring} $\FP(P)$ from the previous section verbatim.
The \Dfn{order polytope} of~$P$ was introduced in~\cite[Definition~1.1]{StanleyTwoPosetPolytopes} as
\begin{align*}
  \mathcal O(P) &= \big\{ (a_p)_{p\in P} \in [0,1]^P \mid a_p \leq a_q \text{ for all } p \preceq q \big\} \subset \R^P\,.
\end{align*}
This polytope encodes many interesting properties of the poset.

\begin{prop}[\!\!{{\cite[Corollary 1.3]{StanleyTwoPosetPolytopes}}}]
\label{prop:orderpolytopevertices}
The vertices of $\mathcal O(P)$ are the indicator vectors of order filters,
\begin{equation*}
  V_P = \big\{ (a_p)_{p \in P} \in \{0,1\}^P \mid a_p = 1 \Rightarrow a_q = 1 \text{ for all } p \preceq q \big\}\,.
\end{equation*}
\end{prop}

A quick calculation using the ideal-variety correspondence (see~\cite[Chapter 4]{cox-little-oshea}, for example) implies that $V_P$ is an \emph{affine variety} and decomposes as
\begin{multline*}
  V_P = \big\{ (z_p)_{p\in P} \in \C^{P} \mid z_p(z_p - 1) = 0 \text{ for all } p \in P \big\} \\\cap 
  \big\{z\in \C^{P} \mid z_p(z_p - z_q) = 0 \text{ for all } p \prec p \big\}\,,
\end{multline*}
where the first condition enforces the property of being zero-one vectors, and the second translates the condition from \Cref{prop:orderpolytopevertices}.
Therefore its vanishing ideal is
\[
  I(V_P) = \sqrt{\big\langle z_p(z_p - 1) \mid p\in P\big\rangle + \big\langle z_p(z_p - z_q) \mid p \prec q\big\rangle}\,.
\]

We aim to show that
\[
  \GGG= \big\{z_p(1 - z_q) \mid p \preceq q\big\}
\]
is a Gröbner basis for~$I(V_P)$ by evoking \Cref{lem:integer-grobner}.

\begin{lem}
\label{prop:containment}
  $\GGG \subseteq I(V_P)$.
\end{lem}

\begin{proof}
  For $p \preceq q$, we calculate
  \[
    z_p(1-z_q) = z_p - z_pz_q = z_p(z_p-z_q)-z_p(z_p-1)\,.
  \]
  Since $z_p(z_p-z_q),z_p(z_p-1) \in I(V_P)$, the statement follows.
\end{proof}

\begin{thm}
\label{thm:C-iso}
  $\GGG$ is a Gröbner basis of $I(V_P)$.
  Moreover,
  \[
    \FP(P) \cong \C[V_P] = \C[z_p \mid p\in P]\big/ \langle \GGG \big\rangle\,.
  \]
\end{thm}

\begin{proof}
Observe that~$|V_P| = \#\Anti_{P}$ is a finite set of points.
It is a standard exercise that its coordinate ring is $\C[V_p] = S/I(V_P) \cong \C^{\#\Anti_{P}}$.
We have seen in \Cref{prop:containment} that $\GGG\subseteq I(V_P)$ and we obtain a surjection 
\[
  f: S/\GGG  \twoheadrightarrow S/I(V_P) \cong \C^{\#\Anti_{P}}\,.
\]
Analogously to the proof of \Cref{prop:fp-vg-generating-set}, any term order~$\omega$ yields that the set of $\init_\omega(\GGG)$-standard monomials is in bijection with antichains in~$P$.
The surjection~$f$ thus satisfies both conditions for \Cref{lem:integer-grobner} and we obtain that~$\GGG$ is a Gröbner basis for~$I(V_P)$.
This implies
\[
  \C[V_P] = \C[z_p \mid p\in P]\big/ \langle \GGG \big\rangle\,.
\]
The proof of $\FP(P) \cong \C[V_P]$ is then verbatim the same as the proof of \Cref{prop:fp-vg-generating-set}.
\end{proof}

\begin{coro}
\label{thm:Zisos}
  For any term order, we have a $\C$-algebra isomorphism
  \begin{align*}
    \mathfrak{gr}\big(\FP(P)\big) & \cong \C[z_p \mid p \in P]\,\big/\, \langle z_pz_q \mid p \preceq q\rangle
  \end{align*}
  where $\mathfrak{gr}\big(\FP(P)\big)$ is the associated graded ring with respect to the degree filtration.
  Moreover,
  \[
    \operatorname{Hilb}\left(\mathfrak{gr}\big(\FP(P)\big),t\right) = \sum_{A\in\Anti_P} t^{\# A}\,.
  \]
\end{coro}

\begin{proof}
  The first part is an immediate consequence of the proof of \Cref{thm:C-iso} using \Cref{lem:integer-grobner} where we use that $-z_pz_q$ is the leading term of $z_p(1-z_q)$ for any term order.
  To obtain the equality for the Hilbert series, observe that the $k$\th\ graded component is linearly generated by the squarefree monomials
    \[
      \left\{ \prod_{p \in A} z_p \mid A \in \Anti_P \text{ with } |A| = k \right\}\,. \qedhere
    \]
\end{proof}

\begin{rem}[A recursion on standard monomials]\label{rem:recursion}
  We have seen in \Cref{thm:Zisos} that the standard monomials for $\FP(P)$ with respect to any term order are identified with the set of antichains in~$P$.
  Since $V_P \subset \{0,1\}^P$, we may on the other hand revisit \cite[Lemma~3.1]{EngstroemSanyalStump}, where a recursive description of the standard monomials is given for the lexicographic term order induced by a total ordering of the variables $\{ z_p \mid p \in P\}$.
  These two viewpoints match as we describe now.
  For ease of notation, we assume that~$P$ is naturally labeled by $P = \{1,\dots,k\}$, meaning that if $i\preceq j$ in~$P$, then $i \leq j$.
  Now let~$P^1 = P \setminus \{k\}$ be the subposet of~$P$ obtained by deleting the element~$k$ (which is necessarily maximal in~$P$) and let~$P^0 = P \setminus \ideal(k) \subseteq P^1 \subseteq P$ be obtained by removing the principal order ideal generated by~$k$.
  The recursive description in \cite[Lemma~3.1]{EngstroemSanyalStump} for the lexicographic term order on $\C[z_1,\dots,z_k]$ induced by $z_1 > z_2 > \dots > z_k$ then becomes the well-known recursive description of order ideals of~$P$ given by
  \[
    \Anti_P = \Anti_{P^1} \cup \big\{ A \cup \{k\} \mid A \in \Anti_{P^0} \big\}
            \,.
  \]
  This equality says that an antichain in~$P$ does either not contain~$k$ and is thus an antichain in $P^1$, or it does contain~$k$ and it thus is the union of antichain in~$P^0$ with $\{k\}$.
\end{rem}

\begin{ex}
\label{sec:longexample2}
  We close this section with an example, for the poset pictured below.
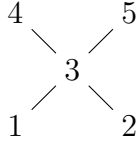
  \begin{center}
  \begin{tikzpicture}[scale=.75]
  \node[invisivertex] (1) at (0,0){1};
  \node[invisivertex] (2) at (2,0){2};
  \node[invisivertex] (3) at (1,1){3};
  \node[invisivertex] (4) at (0,2){4};
  \node[invisivertex] (5) at (2,2){5};
  
  \path[-] (1) edge [] node[above] {} (3);
  \path[-] (2) edge [] node[above] {} (3);
  \path[-] (3) edge [] node[above] {} (4);
  \path[-] (3) edge [] node[above] {} (5);
  \end{tikzpicture}
  \end{center}
  The vertices of its order polytope $\mathcal O(P)$ are
  \begin{align*}
  V_P = \big\{ 00000,00010,00001,00011,00111,10111,01111,11111\big\}.
  \end{align*}
  Its order ring is the coordinate ring of $V_P$,
  \[
    \FP(P) \cong \C[V_P] = \C[z_1,z_2,z_3,z_4,z_5] \big/ \big\langle \GGG\big\rangle\,,
  \]
  $\GGG= \big\{z_p(1 - z_q) \mid p \preceq q\big\}$
  where $\GGG$ consists of $z_i-z_i^2$ for $i \in \{1,\dots,5\}$ as well as of
  \begin{align*}
  z_1(1- z_3), & ~z_1(1- z_4),~z_1(1- z_5),~z_2(1- z_3),~z_2(1- z_4)\,,\\
  & z_2(1- z_4),~z_2(1- z_5),~z_3(1- z_4),~z_3(1- z_5).
  \end{align*}
  The associated graded (with respect to the degree filtration) of~$\FP(P)$ is
  \[
    \grr(\FP(P))
    \cong
    \C \cdot \{1\}
    \oplus \C \cdot \{z_1,\thinspace z_2,\thinspace z_3,\thinspace z_4,\thinspace z_5\}
    \oplus \C \cdot \{z_1\cdot z_2,\thinspace z_4\cdot z_5\}\,.
  \]
  indexed by the antichains of~$P$,
  \[
    \Anti_{P} = \big\{ \emptyset,~\{1\},~\{2\},~\{3\},~\{4\},~\{5\},~\{1,2\},~\{4,5\} \big\}\,.
  \]
  Its Hilbert series finally is $\Hilb(\grr(\FP(P)),t) = 1 + 5t + 2t^2$.
\end{ex}



\bibliographystyle{plain}
\bibliography{bibliog}

\end{document}